\documentclass[11pt,reqno,a4paper]{amsart}

\usepackage[utf8]{inputenc}
\usepackage[T1]{fontenc}
\usepackage{lmodern}
\usepackage[ngerman,english]{babel}
\usepackage{microtype}
\usepackage{color}

\usepackage[pdftex]{graphicx}
\usepackage{latexsym}
\usepackage{amsmath,amssymb,amsthm}
\usepackage{bbm}
\usepackage{mathtools}
\usepackage {enumitem}
\usepackage{tikz,pgfplots}
\pgfplotsset{compat=1.16}

\usepackage[backend=biber,maxbibnames=99]{biblatex}  
\usepackage{csquotes}

\renewbibmacro*{publisher+location+date}{%
	\printlist{publisher}%
	\iflistundef{location}
	{\setunit*{\addcomma\space}}
	{\setunit*{\addcomma\space}}%
	\printlist{location}%
	\setunit*{\addcomma\space}%
	\usebibmacro{date}%
	\newunit}

\setlist[enumerate,1]{noitemsep}
\newtheorem{Satz}{Satz}[section]
\newtheorem{Proposition}[Satz]{Proposition} 
     
\newtheorem{Lemma}[Satz]{Lemma}	
\newtheorem{Theorem}[Satz]{Theorem}	
\newtheorem{Corollary}[Satz]{Corollary}	
\theoremstyle{definition}

\newtheorem{Assumption}[Satz]{Assumption}

\numberwithin{equation}{section} 

\newcommand{\T}{\mathbb{T}} 
\newcommand{\R}{\mathbb{R}} 
\newcommand{\Z}{\mathbb{Z}} 
\newcommand{\N}{\mathbb{N}} 
\newcommand{\eps}{\varepsilon}

\newcommand{\dd}{\, \mathrm{d}}
\newcommand{\iu}{\mathrm{i}}

\makeatletter
\newcommand{\opnorm}{\@ifstar\@opnorms\@opnorm}
\newcommand{\@opnorms}[1]{%
	\left|\mkern-1.5mu\left|\mkern-1.5mu\left|
	#1
	\right|\mkern-1.5mu\right|\mkern-1.5mu\right|
}
\newcommand{\@opnorm}[2][]{%
	\mathopen{#1|\mkern-1.5mu#1|\mkern-1.5mu#1|}
	#2
	\mathclose{#1|\mkern-1.5mu#1|\mkern-1.5mu#1|}
}
\makeatother

\usepackage[open]{bookmark}

\usepackage{hyperref}

\allowdisplaybreaks

\title[Improved error estimates for low-regularity integrators]{Improved error estimates for low-regularity integrators using space-time bounds}
\author{Maximilian Ruff}
\address{Karlsruhe Institute of Technology, Department of Mathematics, Englerstraße 2, 76131 Karlsruhe, Germany}
\email{maximilian.ruff@kit.edu}
\thanks{Funded by the Deutsche Forschungsgemeinschaft (DFG, German Research Foundation) – Project-ID 258734477 – SFB 1173.}
\thanks{I thank Benjamin Dörich and Roland Schnaubelt for fruitful discussions. I also thank the referees for very interesting comments, which led to improvements of the presentation.} 
\subjclass[2020]{Primary 65M15; Secondary 35L71, 35Q55, 65M12}
\keywords{Semilinear Schrödinger equation, semilinear wave equation, low-regularity integrator, error analysis, Strichartz estimates, null form estimate}

\begin{document}
	\selectlanguage{english}
	\pagestyle{headings}
	\begin{abstract}
	We prove optimal convergence rates for certain low-regularity integrators applied to the one-dimensional periodic nonlinear Schrödinger and wave equations under the assumption of $H^1$ solutions. For the Schrödinger equation we analyze the exponential-type scheme proposed by Ostermann and Schratz in 2018, whereas in the wave case we treat the corrected Lie splitting proposed by Li, Schratz, and Zivcovich in 2023. We show that the integrators converge with their full order of one and two, respectively. In this situation only fractional convergence rates were previously known. The crucial ingredients in the proofs are known space-time bounds for the solutions to the corresponding linear problems. More precisely, in the Schrödinger case we use the $L^4$ Strichartz inequality, and for the wave equation a null form estimate. To our knowledge, this is the first time that a null form estimate is exploited in numerical analysis. We apply the estimates for continuous time, thus avoiding potential losses resulting from discrete-time estimates.
	\end{abstract}

	\maketitle
	
	\section{Introduction}

Due to their importance as model problems in mathematical physics, the nonlinear Schrödinger and wave equations have been intensively studied in the past decades, both analytically and numerically. In this work we study their numerical time integration in the one-dimensional cases with periodic boundary conditions. We treat the semilinear Schrödinger equation
\begin{equation}\begin{aligned} \label{NLS} \iu\partial_t u +\partial_x^2 u&=   \mu |u|^2u  , \quad (t,x) \in [0,T] \times \T, \\ 
u(0)&=u_0 \in H^1(\T),
\end{aligned}\end{equation}
where we allow for both signs $\mu \in\{\pm 1\}$. The symbol $\T = \R / \Z$ denotes the one-dimensional torus. Our second problem is the semilinear wave equation 
\begin{equation}\begin{aligned} \label{NLW2} \partial_t^2 u -\partial_x^2 u &= g(u)  , \quad (t,x) \in [0,T] \times \T, \\ 
u(0)&=u_0 \in H^1(\T), \\
\partial_tu(0)&=v_0 \in L^2(\T),
\end{aligned}\end{equation}
with a general nonlinearity $g \in C^2(\R,\R)$. Our regularity assumptions on the initial data are natural in view of the energy conservation laws. 

For the time discretization of dispersive or hyperbolic equations such as \eqref{NLS} and \eqref{NLW2}, \emph{low-regularity integrators} have recently gained a lot of attention in the literature. See, e.g., \cite{BrunedSchratz,CaoSchr,MaierhoferExplSym,KdVLowReg,KleinGordon,KdVUnfiltered,ORS,ORSBourgain,OS,OWY,RS} for some important contributions. In low-regularity settings, these tailor-made integrators can outperform more classical schemes (such as splitting methods \cite{Lubich} or classical exponential integrators \cite{Gauckler,ExpInt}) thanks to an improved local error structure which requires less regularity.

In the present paper we analyze two known schemes of this type, the first-order integrator for \eqref{NLS} from \cite{OS} and the second-order scheme for \eqref{NLW2} from \cite{KleinGordon}. We show that they converge with their full order in particular situations where previously only fractional convergence orders were known. The general outline of proof is in both cases the same. We first derive a suitable representation of the local error. This has the property that in a second step, the sum of the local error terms can be optimally estimated exploiting an equation-specific space-time inequality for the solution $u$. Here we only use the estimates for continuous time, since discrete-time estimates often involve a loss, see \cite{ORS,ORSBourgain,RuffSchnaubelt}. In the case of the wave equation we use a null form estimate, which seems to be a new tool in numerical analysis. The proof of the error bound is then completed in a classical way by a discrete Gronwall argument. 
Hence, our proof strategy is very flexible and could possibly be adapted to show error bounds also for other equations and integrators. 
In this work we only analyze the temporal semi-discretization, but expect that an extension to a fully discrete setting is possible.

\subsection{The Schrödinger case}

In the seminal paper \cite{OS}, a low-regularity integrator was proposed for the time integration of the nonlinear Schrödinger equation \eqref{NLS} (and also its higher-dimensional versions). The scheme computes approximation $u_n \approx u(n\tau)$ via
\begin{equation}\label{LRI} 
u_{n+1} = \Phi_\tau(u_n) \coloneqq e^{\iu \tau \partial_x^2}\Big(u_n-\iu\tau\mu(u_n)^2 \varphi_1(-2\iu\tau\partial_x^2) \bar u_n \Big). \end{equation}
The operator $\varphi_1(-2\iu\tau\partial_x^2)$ can be defined in Fourier space or using the functional calculus for $\varphi_1(z) = (e^z-1)/z$. For our purposes, the definition via the integral representation
\begin{equation}\label{Defphi1}
\varphi_1(-2\iu\tau\partial_x^2)f \coloneqq \frac1\tau \int_0^\tau e^{-2\iu s \partial_x^2} f \dd s
\end{equation}
for $f \in L^2(\T)$ is convenient.
The authors in \cite{OS} proved a general convergence result which in the one-dimensional case reads as follows.
\begin{Theorem}[\cite{OS}]\label{ThmOS}
	Let $r>1/2$ and $\gamma \in (0,1]$. Assume that the solution $u$ to \eqref{NLS} satisfies $u(t) \in H^{r+\gamma}(\T)$ for all $t \in [0,T]$. Then there are a constant $C>0$ and a maximum step size $\tau_0>0$ such that the approximations $u_n$ obtained by \eqref{LRI} satisfy the error bound
	\[ \|u(n\tau)-u_n\|_{H^r(\T)} \le C\tau^\gamma \]
	for all $\tau \in (0,\tau_0]$ and $n \in \N_0$ with $n\tau \le T$. The numbers $C$ and $\tau_0$ only depend on $T$ and $\|u\|_{L^\infty([0,T],H^{r+\gamma}(\T))}$.
\end{Theorem}

We will later make use of Theorem \ref{ThmOS} since it provides an a-priori bound in $L^\infty$ for the numerical solution $u_n$ if $\tau$ is small enough.
The condition $r>1/2$ in Theorem \ref{ThmOS} arises from the use of the algebra property of the Sobolev space $H^r(\T)$. The space $L^2(\T)$ does not have this property, so it is a natural question if Theorem \ref{ThmOS} still holds if $r = 0$ and $\gamma = 1$. In this case, the numerical experiment does not give a clear picture concerning the convergence rate, cf.\ our discussion in Section \ref{SecNumExp}. In the work \cite{ORS}, the problem \eqref{NLS} was considered with spatial domain $\R^d$, $d \in \{1,2,3\}$. The difficulty in the error analysis of the scheme \eqref{LRI} is that its local error is roughly of the form
\[ \tau^2 |\partial_xu|^2u, \]
cf.\ p.731 of \cite{ORS}. Estimating this is $L^2$ for fixed times would require that $u\in W^{1,4}$, which is not covered by the assumption $u \in H^1$. It is however known that solutions to dispersive equations such as \eqref{NLS} enjoy better integrability properties in space if we also involve integration in the time variable. This is formalized using \emph{Strichartz estimates}, which control mixed space-time $L^pL^q$ norms of solutions to linear dispersive equations in terms of the data, cf.\ Chapter 2.3 of \cite{Tao2006}. 
In \cite{ORS}, the authors proved discrete-time Strichartz estimates and used them to show fractional convergence rates (strictly between $1/2$ and $1$ depending on the dimension) in $L^2$ for a frequency-filtered variant of \eqref{LRI}. In the case $d=1$, the convergence rate was $5/6$. In the subsequent paper \cite{ORSBourgain}, the authors analyzed the problem \eqref{NLS} on the torus $\T$. They introduced \emph{discrete Bourgain spaces} and used them to prove a convergence rate of almost $7/8$ for a significantly refined frequency-filtered variant of \eqref{LRI}. 

In \cite{ORS,ORSBourgain} the optimal first-order convergence could not be reached {since} the discrete-time Strichartz and Bourgain space estimates only hold for frequency localized functions {and} contain a multiplicative loss depending on $K\tau^{1/2}$, where $K$ is the largest frequency and $\tau$ denotes the time step-size.
{This is because the regularizing effect described by Strichartz estimates only exists when time is averaged, whereas for fixed times, the Sobolev embedding is generally sharp. Hence, the constant in the discrete-time Strichartz estimates (that by construction contain fixed-time estimates) has to blow up in the limit $K \to \infty$. For scaling reasons, the loss depending on $K\tau^{1/2}$ is optimal, cf.\ the analogous discussion in the case of the wave equation after Theorem 2.3 of \cite{R}. In order to apply the discrete-time estimates to control the terms stemming from the local error, it is tempting to employ the CFL-type condition $K=\tau^{-1/2}$ to avoid this loss. However, as explained in \cite{ORS,ORSBourgain}, the choice $K=\tau^{-1/2}$ would only lead to a global error bound of order $1/2$ since the error caused by frequency filtering is of order $K^{-1}$ if the error is measured in $L^2$ for $H^1$ initial data. In order to balance the contributions of the error terms, the works \cite{ORS,ORSBourgain} proposed to take $K = \tau^{-\alpha/2}$ for a suitable choice of $\alpha>1$, leading to the global error bounds of fractional order between $1/2$ and $1$ mentioned above.}

In this work we extend Theorem \ref{ThmOS} to the case $r=0$ and $\gamma=1$ with optimal first-order convergence. In contrast to \cite{ORS,ORSBourgain}, we do not use frequency filtering and discrete-time Strichartz or Bourgain space estimates. Instead, we derive an error representation which allows us to apply the continuous-time periodic Strichartz estimate
\begin{equation}\label{eq:strich}
	\|e^{\iu t \partial_x^2}f\|_{L^4([0,T] \times \T)} \lesssim_T \|f\|_{L^2(\T)}.
\end{equation}
{This estimate does not suffer from the disadvantages of the discrete-time estimates that were described above.}
A proof of \eqref{eq:strich} can be found in Theorem 1 and the subsequent remark of \cite{Zygmund} or Proposition 2.1 of \cite{Bourgain93}. 

The idea to use continuous-time Strichartz estimates to control the local error goes back to \cite{IgnatSplitting}. {That the avoidance of frequency filtering can lead to improved error bounds was already observed in the context of the Korteweg--de Vries (KdV) equation in \cite{KdVUnfiltered}, where the error analysis is based on different tools such as trilinear estimates for the KdV operator and continuation of the numerical solution into a function on $[0,T]$.}

Our convergence result in $L^2$ reads as follows. Its proof is carried out in Section \ref{SecSchr}.

\begin{Theorem}\label{ThmLRI}
	Assume that the solution $u$ to \eqref{NLS} satisfies $u(t) \in H^1(\T)$ for all $t \in [0,T]$. Then there are a constant $C>0$ and a maximum step size $\tau_0>0$ such that the approximations $u_n$ obtained by \eqref{LRI} satisfy the error bound
	\[ \|u(n\tau)-u_n\|_{L^2(\T)} \le C\tau \]
	for all $\tau \in (0,\tau_0]$ and $n \in \N_0$ with $n\tau \le T$. The numbers $C$ and $\tau_0$ only depend on $T$ and $\|u\|_{L^\infty([0,T],H^1(\T))}$.
\end{Theorem}

We comment on possible extensions of Theorem \ref{ThmLRI} to higher dimensions. The embedding $H^1 \hookrightarrow L^\infty$ as well as the estimate \eqref{eq:strich}, which are both crucially exploited in the proof of Theorem \ref{ThmLRI}, are then wrong, in general. In two dimensions, they however both only require an arbitrary small amount of extra regularity (see Proposition 3.6 of \cite{Bourgain93} for the 2D version of \eqref{eq:strich}). Therefore, it is possible to extend Theorem \ref{ThmLRI} to the 2D case under the slightly stronger regularity assumption ${u_0} \in H^{1+\eps}$ for some $\eps >0$. One could also stick to the $H^1$ assumption if one considers a suitably filtered variant of \eqref{LRI} and lowers the convergence rate by $\eps$. The three-dimensional case seems to be more difficult and we do not know how the optimal result then looks like. The situation becomes easier if the torus $\T^d$ is replaced by the full space $\R^d$, since then a wider range of Strichartz estimates becomes applicable, cf.\ Chapter 2.3 of \cite{Tao2006}. 

The analysis of the present paper could also be adapted to show that the second-order scheme for \eqref{NLS} proposed in \cite{BrunedSchratz,OWY,RS} is second-order convergent for $H^2$ initial data. 
In a similar spirit, it seems possible to extend our analysis to the symmetrized two-step methods that were recently proposed in \cite{MaierhoferExplSym}.

\subsection{The wave case}

For the nonlinear wave equation \eqref{NLW2}, the authors in \cite{KleinGordon} proposed a low-regularity integrator which was called the \emph{corrected Lie splitting}. It computes approximations $(u_n,v_n) \approx (u(n\tau),\partial_tu(n\tau))$ via
\begin{equation}\label{CorrLie}
\begin{pmatrix}u_{n+1} \\ v_{n+1}\end{pmatrix} = e^{\tau A} \Bigg[\begin{pmatrix}u_n \\ v_n\end{pmatrix} +\tau\begin{pmatrix}0\\ g(u_n)\end{pmatrix} + \tau^2 \varphi_2(-2\tau A)\begin{pmatrix}-g(u_n) \\ g'(u_n)v_n\end{pmatrix} \Bigg], 
\end{equation}
with wave operator {$A(u,v)^\top = (v,\partial_x^2 u)^\top$}.
The operator $\varphi_2(-2\tau A)$ is defined by the integral representation
\begin{equation}\label{Defphi2}
\varphi_2(-2\tau A)w \coloneqq \frac{1}{\tau^2}\int_0^\tau (\tau-s) e^{-2s A} w \dd s 
\end{equation}
for $w \in H^1 \times L^2$. Similar as in the Schrödinger case, one could equivalently use the functional calculus for $\varphi_2(z) = (e^z-z-1)/z^2$.
In \cite{KleinGordon}, under a Lipschitz condition on the nonlinearity $g$, it was shown that the scheme \eqref{CorrLie} converges with order $2$ in $H^1 \times L^2$ under the regularity assumption ${(u_0,v_0)} \in H^{1+d/4} \times H^{d/4}$ for spatial dimensions $d \in \{1,2,3\}$. The reason for this additional regularity requirement is that the main part of the local error is roughly of the form
\begin{equation}\label{eq:nullcond}
\|(\partial_tu)^2-\nabla u \cdot \nabla u\|_{L^2(\T^d)},
\end{equation}
cf.\ equation (2.26) of \cite{KleinGordon}. This term was then estimated (for fixed times) using the triangle inequality and the Sobolev embedding $H^{d/4}(\T^d) \hookrightarrow L^4(\T^d)$. 
For the one-dimensional case $d=1$, the authors in \cite{KleinGordon} also gave a convergence result under the weaker regularity assumption ${(u_0,v_0)} \in H^1 \times L^2$. Using an interpolation argument, it was shown that the scheme \eqref{CorrLie} converges almost with order $4/3$ in $H^1 \times L^2$. However, the numerical experiments in \cite{KleinGordon} suggested that the convergence is of order $2$ also in this case. {We also refer to \cite{DiscSol} for a fully discrete error analysis of a frequency-filtered version of the corrected Lie splitting at the even lower regularity $(u_0,v_0) \in H^\gamma \times H^{\gamma-1}$ for $\gamma>0$ in dimensions $d \in \{1,2\}$, where reduced convergence rates were shown under a CFL-type condition.}

Here, we give a rigorous proof of {the} second-order convergence {observed in \cite{KleinGordon} for $H^1 \times L^2$ initial data in one spatial dimension}. In contrast to the Schrödinger case, the 1D wave equation does not exhibit dispersive behavior. Instead, the idea is to exploit that the expression \eqref{eq:nullcond} contains a so-called \emph{null form} which allows for improved  space-time bounds compared to the above fixed-time approach. Such null form estimates are widely used in the analysis of nonlinear wave equations, cf.\ \cite{KlainermanMachedon} or {Chapter 6} of \cite{Tao2006}. They rely on cancellation of parallel interactions (where waves move together) in the bilinear expression in \eqref{eq:nullcond}. In the one-dimensional case one has the following estimate. If $\phi$ solves the linear inhomogeneous wave equation 
$\partial_t^2 \phi - \partial_x^2 \phi = F$
on $[0,T] \times \T$, then one has the inequality
\begin{align}\label{eq:nullformest} &\|(\partial_t\phi)^2-(\partial_x\phi)^2\|_{L^2([0,T] \times \T)} \nonumber\\
&\lesssim_T \|\partial_x\phi(0)\|^2_{L^2(\T)} + \|\partial_t\phi(0)\|^2_{L^2(\T)} + \|F\|^2_{L^1([0,T],L^2(\T))}. 
\end{align}
Note that the right-hand side of \eqref{eq:nullformest} only contains the $L^2$ norm of $\partial_{t,x} \phi(0)$ instead of the $L^4$ norm that would result from the triangle inequality approach. If we replace $\T$ with $\R$, estimate \eqref{eq:nullformest} can be found in (1.8) of \cite{KlainermanMachedon} (in a simplified form) or in (6.29) of \cite{Tao2006}. For convenience, we give a direct proof of \eqref{eq:nullformest} on $\T$ based on d'Alembert's formula at the beginning of Section \ref{SecWave}.

With the help of estimate \eqref{eq:nullformest}, we are able to show the following improved error bound for the corrected Lie splitting \eqref{CorrLie}. The proof is given in Section \ref{SecWave}. To our knowledge, this is the first time that a null form estimate like \eqref{eq:nullformest} is used in numerical analysis.
\begin{Theorem}\label{ThmCorrLie}
	Assume that the solution $u$ to \eqref{NLW2} satisfies $(u(t), \partial_tu(t)) \in H^1(\T) \times L^2(\T)$ for all $t \in [0,T]$. Then there are a constant $C>0$ and a maximum step size $\tau_0>0$ such that the approximations $(u_n,v_n)$ obtained by \eqref{CorrLie} satisfy the error bound
	\[ \|u(n\tau)-u_n\|_{H^1(\T)} + \|\partial_tu(n\tau)-v_n\|_{L^2(\T)} \le C\tau^2 \]
	for all $\tau \in (0,\tau_0]$ and $n \in \N_0$ with $n\tau \le T$. The numbers $C$ and $\tau_0$ only depend on $g$, $T$, $\|u\|_{L^\infty([0,T],H^1(\T))}$, and $\|\partial_tu\|_{L^\infty([0,T],L^2(\T))}$.
\end{Theorem}

The higher-dimensional versions of the null form estimate \eqref{eq:nullformest} require more regularity, cf.\ \cite{KlainermanMachedon} or inequality (6.29) of \cite{Tao2006}. In two dimensions, they could possibly still be used to show an analogue of Theorem \ref{ThmCorrLie} with a convergence rate greater than one under a suitable growth condition on $g$. Very recently, convergence rates for a Strang splitting scheme for the 3D semilinear wave equation with power nonlinearity under the assumption ${(u_0,v_0)} \in H^1 \times L^2$ were obtained in \cite{R}. We do not know whether in this situation it is possible to show higher rates by using a low-regularity integrator instead. 

\medskip

\textbf{Notation.} We use the notation $A \lesssim_\gamma B$ if there is a constant $c>0$ (depending on quantities $\gamma$) such that $A \le cB$. If it is clear from the context, we often abbreviate the Sobolev spaces $H^r = H^r(\T)$ as well as the Lebesgue spaces $L^p = L^p(\T)$. 
{The symbol $\N_0$ denotes the set of natural numbers including zero.}
For a step size $\tau >0$ and a number $n \in \N_0$, the discrete times are denoted by $t_n \coloneqq n\tau$. 
\section{Proof of the result for the nonlinear Schrödinger equation}\label{SecSchr}

In this section we prove Theorem \ref{ThmLRI}. We start by converting the linear estimate \eqref{eq:strich} into a bound for the solution $u$ to the nonlinear problem \eqref{NLS}.

\begin{Assumption}\label{AssSchr}
	There exists a time $T>0$ and a solution $u \in C([0,T], H^1) \cap C^1([0,T],H^{-1})$ to the nonlinear Schrödinger equation \eqref{NLS} 
	with bound
	\begin{equation*}\label{DefMSchr}
	M \coloneqq \|u\|_{L^\infty([0,T], H^1)}.
	\end{equation*}
\end{Assumption}

\begin{Corollary}\label{CorStrichu}
	Let $u$, $T$, and $M$ be given by Assumption \ref{AssSchr}. Then we have the estimate
	\begin{equation*}
	\|\partial_x u\|_{L^4([0,T] \times \T)}  \lesssim_{M,T} 1. 
	\end{equation*}
\end{Corollary}
\begin{proof}
	We apply estimate \eqref{eq:strich} to Duhamel's formula
	\[ u(t) = e^{\iu t \partial_x^2} u_0 -\iu \mu \int_0^t e^{\iu (t-s) \partial_x^2} (|u|^2u)(s) \dd s. \]
	Using also Minkowski's inequality and Sobolev's embedding $H^1 \hookrightarrow L^\infty$, we get
	\begin{align*}
	\|\partial_x u\|_{L^4([0,T] \times \T)}  &\lesssim_T \|\partial_x u_0\|_{L^2} + \int_0^T\|\partial_x(|u|^2u)(s)\|_{L^2} \dd s \\
	&\lesssim_M 1 + \|u\|_{L^2([0,T],L^\infty)}^2\|\partial_x u\|_{L^\infty([0,T],L^2)} \lesssim_{M,T} 1. \qedhere
	\end{align*}
\end{proof}

We now give a representation of the local error of the low-regularity integrator \eqref{LRI}. The calculations are inspired by the ones in Section 3 of \cite{RS}. But compared to there and also \cite{OS}, we do not insert the approximation $u(s) \approx e^{\iu s \partial_x^2}u_0$ at first. This makes it easier for us to apply Corollary \ref{CorStrichu} in the subsequent Lemma \ref{LemSumErrSchr}. 

\begin{Lemma}\label{LemLokFSchr}
	Let $u$ and $T$ be given by Assumption \ref{AssSchr}. Then for $\tau \in (0,T]$, the local error of \eqref{LRI} is given by
	\[ u(\tau)-u_1 = \mu \int_0^\tau\int_0^s e^{\iu (\tau-\sigma)\partial_x^2} D(\sigma,s) \dd \sigma \dd s. \]
	Here we define
	\[D(\sigma,s) = D_1(\sigma,s) + D_2(\sigma,s) + D_3(\sigma,s) \]
	with 
	\begin{align*}
	D_1(\sigma,s) &\coloneqq \mu u(\sigma)^2\Big(e^{2\iu(\sigma-s)\partial_x^2}(|u|^2\bar u)(\sigma) -2 |u(\sigma)|^2 e^{2\iu(\sigma-s)\partial_x^2} \bar u(\sigma) \Big), \\
	D_2(\sigma,s) &\coloneqq -2 (\partial_x u(\sigma))^2  e^{2\iu(\sigma-s)\partial_x^2} \bar u(\sigma), \\
	D_3(\sigma,s) &\coloneqq -4 u(\sigma)\partial_x u(\sigma)    e^{2\iu(\sigma-s)\partial_x^2} \partial_x \bar u(\sigma).
	\end{align*}
\end{Lemma}	

\begin{proof}
	By \eqref{Defphi1}, we have
	\[ \tau\varphi_1(-2\iu \tau \partial_x^2 ) \bar u_0 = \int_0^\tau e^{-2\iu s \partial_x^2} \bar u_0 \dd s.
	\]	
	Hence, we get by Duhamel's formula and the fundamental theorem of calculus
	\begin{align*}
	u(\tau) - u_1 &= -\iu \mu e^{\iu \tau \partial_x^2} \int_0^\tau \Big(e^{-\iu s \partial_x^2} (u^2\bar u)(s)  - u_0^2 e^{-2\iu s \partial_x^2 } \bar u_0 \Big) \dd s \\
	&= \mu e^{\iu \tau \partial_x^2}\!\int_0^\tau  (N(s,s)-N(0,s)) \dd s = \mu e^{\iu \tau \partial_x^2}\int_0^\tau \int_0^s \partial_1 N(\sigma,s) \dd \sigma \dd s.
	\end{align*}
	Here, the function $N(\cdot,s) \in C^1([0,\tau],H^{-1}(\T))$ is defined as
	\[ N(\sigma,s) \coloneqq -\iu e^{-\iu \sigma \partial_x^2}\Big(u(\sigma)^2 e^{2 \iu (\sigma-s)\partial_x^2} \bar u(\sigma)\Big). \]
	Using the product rule 
	and the differential equation \eqref{NLS}, we compute the derivative as
	\begin{align*}
	\partial_1 N(\sigma,s) &= e^{-\iu \sigma \partial_x^2} \Big[ -\partial_x^2\Big( u(\sigma)^2 e^{2 \iu (\sigma-s)\partial_x^2} \bar u(\sigma)\Big) -2\iu u(\sigma)  \partial_t u(\sigma)e^{2 \iu (\sigma-s)\partial_x^2} \bar u(\sigma) \\
	&\quad  + 2u(\sigma)^2 e^{2 \iu (\sigma-s)\partial_x^2} \partial_x^2 \bar u(\sigma) - \iu u(\sigma)^2 e^{2 \iu (\sigma-s)\partial_x^2} \partial_t \bar u(\sigma) \Big] \\
	&= e^{-\iu \sigma \partial_x^2} \Big[\! -2 \partial_x^2 u(\sigma) u(\sigma)  e^{2 \iu (\sigma-s)\partial_x^2} \bar u(\sigma) -2 (\partial_x u(\sigma))^2 e^{2 \iu (\sigma-s)\partial_x^2} \bar u(\sigma) \\
	&\quad  - 4 u(\sigma ) \partial_x u(\sigma)  e^{2 \iu (\sigma-s)\partial_x^2} \partial_x\bar u(\sigma) - u(\sigma)^2 e^{2 \iu (\sigma-s)\partial_x^2}\partial_x^2 \bar  u(\sigma) \\
	&\quad  + 2 u(\sigma) \partial_x^2 u(\sigma) e^{2 \iu (\sigma-s)\partial_x^2} \bar u(\sigma) - 2\mu   u(\sigma) (|u|^2u)(\sigma) e^{2 \iu (\sigma-s)\partial_x^2} \bar u(\sigma) \\
	&\quad +2 u(\sigma)^2 e^{2 \iu (\sigma-s)\partial_x^2} \partial_x^2 \bar u(\sigma) - u(\sigma)^2 e^{2 \iu (\sigma-s)\partial_x^2} \partial_x^2 \bar u(\sigma) \\
	&\quad + \mu u(\sigma)^2 e^{2 \iu (\sigma-s)\partial_x^2} (|u|^2 \bar u)(\sigma) \Big] \\
	&=  e^{-\iu \sigma \partial_x^2} \Big[ -2 (\partial_x u(\sigma))^2 e^{2 \iu (\sigma-s)\partial_x^2} \bar u(\sigma) \\
	&\quad   - 4 u(\sigma ) \partial_x u(\sigma)  e^{2 \iu (\sigma-s)\partial_x^2} \partial_x \bar u(\sigma) \\
	&\quad   + \mu   u(\sigma)^2 \Big(-2 |u(\sigma)|^2 e^{2 \iu (\sigma-s)\partial_x^2} \bar u(\sigma) +  e^{2 \iu (\sigma-s)\partial_x^2} (|u|^2 \bar u)(\sigma)\Big) \Big] \\
	&= e^{-\iu \sigma \partial_x^2} \Big[D_1(\sigma,s)+ D_2(\sigma,s) + D_3(\sigma,s)\Big] ,
	\end{align*}
	where we exploit the cancellation of all second-order partial derivatives. The derivative is well-defined in $H^{-1}(\T)$ since in 1D we can use the embedding $L^1 \hookrightarrow H^{-1}$ and that the multiplication by an $H^1$ function is a continuous operator on $H^{-1}$.
\end{proof}

In the next step we bound the sum of the local errors terms, where we will crucially exploit Corollary \ref{CorStrichu} as well as the dual of estimate \eqref{eq:strich}. 

\begin{Lemma}\label{LemSumErrSchr}
	Let $u$, $T$, and $M$ be given by Assumption \ref{AssSchr}.
	Then we can bound the sum of local errors of \eqref{LRI} by
	\[ \Big\|\sum_{k=0}^{n-1} e^{\iu (n-k-1)\tau \partial_x^2 }\Big(u(t_{k+1})-\Phi_\tau(u(t_k))\Big) \Big\|_{L^2} \lesssim_{M,T} \tau,  \]
	for all $\tau \in (0,T]$ and $n \in \N_0$ with $n\tau \le T$.
\end{Lemma}

\begin{proof}
	By Lemma \ref{LemLokFSchr} with $u(t_k+\cdot)$ instead of $u$,
	\begin{align*}
	 &\sum_{k=0}^{n-1} e^{\iu (n-k-1)\tau \partial_x^2 }\Big(u(t_{k+1})-\Phi_\tau(u(t_k))\Big) \\
	 &= \mu \sum_{k=0}^{n-1} e^{\iu (n-k)\tau \partial_x^2 } \int_0^\tau \int_0^s e^{-\iu \sigma \partial_x^2}D(t_k + \sigma,t_k + s) \dd \sigma \dd s. 
	 \end{align*}
	 We now use the decomposition $D = D_1 + D_2 + D_3$ from Lemma \ref{LemLokFSchr}. For the first term we even get
	 \[ \Big\|\sum_{k=0}^{n-1} e^{\iu (n-k)\tau \partial_x^2 } \int_0^\tau \int_0^s e^{-\iu \sigma \partial_x^2}D_1(t_k + \sigma,t_k + s) \dd \sigma \dd s \Big\|_{H^1} \lesssim_{M,T} n\tau^2 \lesssim_T \tau, \]
	 using the algebra property of $H^1 $ in 1D. The second term is controlled by
	 \begin{align*}
	 &\Big\|\sum_{k=0}^{n-1} e^{\iu (n-k)\tau \partial_x^2 } \int_0^\tau \int_0^s e^{-\iu \sigma \partial_x^2}D_2(t_k + \sigma,t_k + s) \dd \sigma \dd s \Big\|_{L^2} \\
	 &\le \sum_{k=0}^{n-1} \int_0^\tau \int_0^s \|D_2(t_k + \sigma,t_k + s)\|_{L^2} \dd \sigma \dd s \\
	 &\lesssim \sum_{k=0}^{n-1} \int_0^\tau \int_0^\tau \|(\partial_x u(t_k + \sigma))^2\|_{L^2}\|e^{2\iu (\sigma-s)\partial_x^2}\bar u(t_k + \sigma)\|_{L^\infty} \dd \sigma \dd s \\
	 &\lesssim \tau  \sum_{k=0}^{n-1} \int_0^\tau  \|\partial_x u(t_k + \sigma)\|_{L^4}^2\| u(t_k + \sigma)\|_{H^1} \dd \sigma \lesssim_M \tau \|\partial_x u\|_{L^2([0,T],L^4)}^2  \\
	 &\lesssim_T \tau \|\partial_x u\|_{L^4([0,T]\times \T)}^2 \lesssim_{M,T} \tau.
	 \end{align*}
	 Here we used Hölder's inequality, the Sobolev embedding $H^1 \hookrightarrow L^\infty$, and Corollary \ref{CorStrichu}. The term involving $D_3$ {cannot be dealt with in such a straightforward manner since the Schrödinger propagator $e^{2\iu (\sigma-s)\partial_x^2}$ is not uniformly bounded on $L^4$. Therefore, we first use a change of variables to rewrite this term}
	 as
	  \begin{align*}
	 &\sum_{k=0}^{n-1} e^{\iu (n-k)\tau \partial_x^2 } \int_0^\tau \int_0^s e^{-\iu \sigma \partial_x^2}D_3(t_k + \sigma,t_k + s) \dd \sigma \dd s  \\
	&= \sum_{k=0}^{n-1} \int_{t_k}^{t_{k+1}} \int_{t_k}^s e^{\iu (n\tau-\sigma) \partial_x^2}D_3( \sigma,s) \dd \sigma \dd s \\
	 &= \int_0^{n\tau}e^{\iu (n\tau-\sigma) \partial_x^2} \int_\sigma^{\lceil\frac\sigma\tau\rceil\tau} D_3( \sigma,s)\dd s \dd \sigma,
	 \end{align*}
	 where the application of Fubini's theorem is justified since the double integral converges absolutely in $H^{-1}$. We next apply the dual of the periodic Strichartz estimate \eqref{eq:strich}, which reads
	 \[ \Big\|\int_0^Te^{-\iu t \partial_x^2} F(t) \dd t \Big\|_{L^2} \lesssim_T \|F\|_{L^{\frac43}([0,T] \times \T)}, \]
	 {see Corollary 2.5 of \cite{Bourgain93} and compare also equation (2.25) of \cite{Tao2006}.}
	 We infer that
	 \begin{align*}
	 &\Big\|\int_0^{n\tau}e^{\iu (n\tau-\sigma) \partial_x^2} \int_\sigma^{\lceil\frac\sigma\tau\rceil\tau} D_3( \sigma,s)\dd s \dd \sigma\Big\|_{L^2} \\
	 &\lesssim_T \Big\|\sigma \mapsto \int_\sigma^{\lceil\frac\sigma\tau\rceil\tau} D_3( \sigma,s)\dd s \Big\|_{L^\frac43([0,T] \times \T)} \\
	 & \lesssim \Big\|\sigma \mapsto \|u(\sigma)\|_{L^\infty}\|\partial_x u(\sigma)\|_{L^4} \int_\sigma^{\lceil\frac\sigma\tau\rceil\tau}    \|e^{2\iu(\sigma-s)\partial_x^2} \partial_x \bar u(\sigma)\|_{L^2}\dd s \Big\|_{L^\frac43([0,T])} \\
	 &\lesssim_M \tau \|\partial_x u\|_{L^\frac43([0,T],L^4)}   \lesssim_T \tau  \|\partial_x u\|_{L^4([0,T]\times \T)} \lesssim_{M,T} \tau,
	 \end{align*}
	 using again Hölder's inequality, the Sobolev embedding $H^1 \hookrightarrow L^\infty$, and Corollary \ref{CorStrichu}.
\end{proof}

We can now finish the proof of the global error bound in a classical way with the help of the discrete Gronwall lemma.

\begin{proof}[Proof of Theorem \ref{ThmLRI}]
	We define the error
	\[ e_n \coloneqq u(t_n)-u_n.\]
	We get the recursion formula
	\begin{align*}
	e_{n+1} &= u(t_{n+1})-\Phi_\tau(u(t_n))+\Phi_\tau(u(t_n))-\Phi_\tau(u_n)\\
	&= u(t_{n+1})-\Phi_\tau(u(t_n)) + e^{\iu\tau \partial_x^2}e_n \\
	&\quad-\iu\tau\mu e^{\iu\tau \partial_x^2}\Big((u(t_n))^2 \varphi_1(-2\iu\tau\partial_x^2) \bar u(t_n) -(u_n)^2 \varphi_1(-2\iu\tau\partial_x^2) \bar u_n  \Big).
	\end{align*}
	This formula implies that
	\begin{align*}
	e_n &= \sum_{k=0}^{n-1} e^{\iu(n-k-1)\tau\partial_x^2}\Big( u(t_{k+1})-\Phi_\tau(u(t_k))\Big) \\ 
	&\quad- \iu \tau \mu \sum_{k=0}^{n-1} e^{\iu(n-k)\tau\partial_x^2}\Big((u(t_k))^2 \varphi_1(-2\iu\tau\partial_x^2) \bar u(t_k) -(u_k)^2 \varphi_1(-2\iu\tau\partial_x^2) \bar u_k  \Big),
	\end{align*}
	exploiting that $e_0=0$. From Lemma \ref{LemSumErrSchr} and standard estimates we infer that
	\[\|e_n\|_{L^2} \lesssim_{M,T} \tau +  \tau\sum_{k=0}^{n-1}(1+\|e_k\|^2_{H^\frac34})\|e_k\|_{L^2} \]
	using the Sobolev embedding $H^{3/4} \hookrightarrow L^\infty$ and the representation $u_k = u(t_k)-e_k$. By Theorem \ref{ThmOS} with $r = 3/4$ and $\gamma=1/4$, we find $\tau_0>0$ depending only on $M$ and $T$, such that
	\[\|e_n\|_{H^\frac34} \le 1\]
	for all $\tau \in(0,\tau_0]$ and $n \in \N_0$ with $n\tau \le T$. For such $\tau $ and $n$ we thus derive that
	\[\|e_n\|_{L^2} \lesssim_{M,T} \tau +  \tau\sum_{k=0}^{n-1}\|e_k\|_{L^2}. \]
	By the discrete Gronwall inequality, we can conclude that
	\[ \|e_n\|_{L^2} \lesssim_{M,T} \tau. \qedhere \]
\end{proof}	

\section{Proof of the result for the nonlinear wave equation}\label{SecWave}

In this section we carry out the proof of Theorem \ref{ThmCorrLie}. Our first goal is to show estimate \eqref{eq:nullformest}.
Therefore we define the bilinear form
\[ Q(\phi,\psi) \coloneqq \partial_t\phi\partial_t\psi - \partial_x\phi\partial_x\psi.\]
As a preparatory step, we treat the homogeneous problem.

\begin{Lemma}\label{LemNullfHom}
	Let $\phi$ and $\psi$ solve the homogeneous wave equations
	\begin{align*}
	\partial_t^2 \phi - \partial_x^2 \phi &= 0, \quad \phi(0) = \phi_0, \quad \partial_t \phi(0) = \phi_1 \\
	\partial_t^2 \psi - \partial_x^2 \psi &= 0, \quad \psi(0) = \psi_0, \quad \partial_t \psi(0) = \psi_1
	\end{align*}
	with Cauchy data $\phi_0$, $\psi_0 \in H^1(\T)$ and $\phi_1$, $\psi_1 \in L^2(\T)$. We then have the estimate
	\[ \|Q(\phi,\psi)\|_{L^2(\T \times \T)} \lesssim (\|\partial_x\phi_0\|_{L^2} + \|\phi_1\|_{L^2})(\|\partial_x\psi_0\|_{L^2} + \|\psi_1\|_{L^2}). \]
\end{Lemma}
\begin{proof}
	We define
	\[ v_\phi = \frac12(\partial_x\phi_0+\phi_1), \quad w_\phi = \frac12(\partial_x\phi_0-\phi_1), \]
	and similarly for $\psi$. By d'Alembert's formula, we can then write
	\[\partial_t\phi(t,x)= v_\phi(x+t)-w_\phi(x-t), \quad \partial_x\phi(t,x)= v_\phi(x+t)+w_\phi(x-t).\]
	We compute
	\begin{align*}
	Q(\phi,\psi)(t,x) &= (v_\phi(x+t)-w_\phi(x-t))(v_\psi(x+t)-w_\psi(x-t)) \\
	&\quad - (v_\phi(x+t)+w_\phi(x-t))(v_\psi(x+t)+w_\psi(x-t)) \\
	 &= -2v_\phi(x+t)w_\psi(x-t) -2 w_\phi(x-t)v_\psi(x+t). 
	\end{align*}
	Note that the ``parallel interactions'' (where one has twice ``$x+t$'' or twice ``$x-t$'') are canceled and only the ``transverse interactions'' (where one has once ``$x+t$'' and once ``$x-t$'') remain. See p.293 of \cite{Tao2006} for further explanations of this phenomenon that also apply to the higher dimensional cases. We can now obtain the desired estimate in the following way. By integral substitutions $x-t =y$ and $y+2t = s$, it follows that
	\[ \int_\T \int_\T |v(x+t)w(x-t)|^2 \dd x \dd t = \|v\|^2_{L^2(\T)}\|w\|^2_{L^2(\T)} \]
	for general functions $v$, $w$. Hence,
	\begin{align*}
	\|Q(\phi,\psi)\|_{L^2(\T \times \T)} &\lesssim \|v_\phi\|_{L^2}\|w_\psi\|_{L^2} + \|w_\phi\|_{L^2}\|v_\psi\|_{L^2} \\
	& \lesssim (\|\partial_x\phi_0\|_{L^2} + \|\phi_1\|_{L^2})(\|\partial_x\psi_0\|_{L^2} + \|\psi_1\|_{L^2}). \qedhere
	\end{align*}
\end{proof}

Now we give the proof of \eqref{eq:nullformest}.

\begin{Proposition}\label{PropNullf}
	Let $T>0$ and $\phi$ solve the inhomogeneous wave equation
	\begin{equation}\label{eq:InhWave}
	\partial_t^2 \phi - \partial_x^2 \phi = F, \quad \phi(0) = \phi_0, \quad \partial_t \phi(0) = \phi_1
	\end{equation}
	on $[0,T] \times \T$ with data $\phi_0 \in H^1(\T)$, $\phi_1\in L^2(\T)$, and $ F \in L^1([0,T],L^2(\T))$. Then we have the inequality
	\[ \|Q(\phi,\phi)\|_{L^2([0,T] \times \T)} \lesssim_T \|\partial_x\phi_0\|^2_{L^2} + \|\phi_1\|^2_{L^2} + \|F\|^2_{L^1([0,T],L^2)}. \]
\end{Proposition}
\begin{proof}
	We decompose $\phi = \phi_\mathrm{hom} + \phi_\mathrm{inh}$, where $\phi_\mathrm{hom}$ solves \eqref{eq:InhWave} with $F=0$ and $\phi_\mathrm{inh}$ solves \eqref{eq:InhWave} with $\phi_0=\phi_1=0$. The estimate for $Q(\phi_\mathrm{hom},\phi_\mathrm{hom})$ follows directly from Lemma \ref{LemNullfHom} and the periodicity of $\partial_{t,x}\phi_\mathrm{hom}$ in time. To treat the inhomogeneous part, for almost all $s \in [0,T]$, we define $\phi^s$ to be the solution to the homogeneous equation
	\[
	\partial_t^2 \phi^s - \partial_x^2 \phi^s = 0, \quad \phi^s(s) = 0, \quad \partial_t \phi^s(s) = F(s).
	\]
	By Duhamel's formula, $\phi_\mathrm{inh}$ is then given by
	\[\phi_\mathrm{inh}(t) = \int_0^t \phi^s(t) \dd s. \]
	It follows that we can express the bilinear term as
	\[Q(\phi_\mathrm{inh},\phi_\mathrm{inh})(t) = \int_0^t\int_0^t Q(\phi^s,\phi^r)(t) \dd s \dd r. \]
	From Minkowski's inequality, Lemma \ref{LemNullfHom}, and the energy equality we thus get
	\begin{align*}
	\|Q(\phi_\mathrm{inh},\phi_\mathrm{inh})\|_{L^2([0,T] \times \T)} &\le \int_0^T\int_0^T \|Q(\phi^s,\phi^r)\|_{L^2([0,T] \times \T)} \dd s \dd r \\
	&\lesssim \int_0^T\int_0^T (\|\partial_x\phi^s(0)\|_{L^2} + \|\partial_t\phi^s(0)\|_{L^2}) \\
	&\qquad \qquad\cdot (\|\partial_x\phi^r(0)\|_{L^2} + \|\partial_t\phi^r(0)\|_{L^2}) \dd s \dd r \\
	& \lesssim \|F\|_{L^1([0,T],L^2)}^2.
	\end{align*}
	The mixed term $Q(\phi_\mathrm{hom},\phi_{\mathrm{inh}})$ is treated similarly.
\end{proof}

We now turn to the nonlinear wave equation \eqref{NLW2}. Here it is convenient to work with the first-order reformulation. With the definitions
\begin{equation*}\label{DefsNLW}
U \coloneqq \begin{pmatrix}u \\ v\end{pmatrix} \mathrel{\widehat{=}} \begin{pmatrix}u \\ \partial_t u\end{pmatrix},\ A \coloneqq \begin{pmatrix}0 & I \\ \partial_x^2 & 0 \end{pmatrix},\ G(U) \coloneqq \begin{pmatrix}0 \\ g(u)\end{pmatrix},\ U_0 \coloneqq \begin{pmatrix}u_0 \\ v_0\end{pmatrix},
\end{equation*}
we obtain the equivalent differential equation
\begin{equation}
\begin{aligned} \label{NLW} \partial_{t} U(t) &= AU(t) +G(U(t)), \quad t \in [0,T], \\ U(0)&=U_0.
\end{aligned}\end{equation}
\begin{Assumption}\label{AssWave}
	There exists a time $T>0$ and a solution $U = (u,\partial_tu) \in C([0,T], H^1 \times L^2) \cap C^1([0,T],L^2 \times H^{-1})$ to the nonlinear wave equation \eqref{NLW} 
	with bound
	\begin{equation*}\label{DefMWave}
	M \coloneqq \|U\|_{L^\infty([0,T], H^1 \times L^2)}.
	\end{equation*}
\end{Assumption}

Since the nonlinearity $g$ belongs to $C^2(\R,\R)$, we can find an increasing function $L$ such that the bound
\begin{equation}\label{eq:gbound}
|g(z)|+|g'(z)|+|g''(z)| \le L(|z|)
\end{equation} 
holds for all $z \in \R$. 
In the following, we suppress the dependency on the function $L$ from \eqref{eq:gbound} in the $\lesssim$ notation. We now apply Proposition \ref{PropNullf} to the solution $u$ to the nonlinear problem \eqref{NLW2}.

\begin{Corollary}\label{CorNullEstu}
	Let $u$, $T$, and $M$ be given by Assumption \ref{AssWave}. Then we have the estimate
	\begin{equation*}
	\|(\partial_tu)^2 -(\partial_xu)^2\|_{L^2([0,T] \times \T)}  \lesssim_{M,T} 1. 
	\end{equation*}
\end{Corollary}
\begin{proof}
	By Proposition \ref{PropNullf}, we only need that
	\[ \|g(u)\|_{L^1([0,T],L^2)} \lesssim_{M,T} 1, \]
	and this follows from \eqref{eq:gbound}, Hölder's inequality, and the Sobolev embedding $H^1 \hookrightarrow L^\infty$.
\end{proof}

We now give a brief derivation of the corrected Lie splitting \eqref{CorrLie} proposed in \cite{KleinGordon}. It is based of the Lie splitting approximation for \eqref{NLW}, which is a formally first-order scheme given by
\begin{equation}\label{Lie}
U^\mathrm{Lie}_{n+1} = e^{\tau A}[U_n^\mathrm{Lie} + \tau G(U_n^\mathrm{Lie})].
\end{equation}
By the Duhamel formulation of \eqref{NLW}, the fundamental theorem of calculus, and Fubini's theorem, the local error of \eqref{Lie} can be represented as
\begin{equation}\label{LocErrLie}
U(\tau)-U^\mathrm{Lie}_1 
= e^{\tau A} \int_0^\tau (\tau-s) e^{-s A} H(U(s)) \dd s.
\end{equation}
Here we use the definition
\[H(U(s)) \coloneqq e^{sA} \frac{\mathrm d}{\mathrm ds}\Big[e^{-sA}G(U(s))\Big] = \begin{pmatrix}
-g(u(s)) \\ g'(u(s))\partial_t u(s)
\end{pmatrix}.\] 
Similar as in the Schrödinger case, we do not insert the approximation $U(s) \approx e^{s A}U_0$ (which was used in \cite{KleinGordon}) in order to create better conditions for applying Corollary \ref{CorNullEstu} later.

The crucial observation which allows the construction of the low-regularity integrator is that $H(U)$ satisfies the differential equation
\begin{equation}\label{HDiffEq}
\frac{\mathrm d}{\mathrm ds}H(U(s)) = -AH(U(s))+B(U(s))
\end{equation}
in $L^2(\T) \times H^{-1}(\T)$, where the remainder
\begin{equation}\label{DefB}
B(U) \coloneqq  \begin{pmatrix}
0 \\ g''(u)[(\partial_tu)^2-(\partial_xu)^2]+g'(u)g(u)
\end{pmatrix}
\end{equation}
only contains first-order derivatives of $u$. 
We plug the Duhamel approximation $H(U(s)) \approx e^{-sA}H(U_0)$ for \eqref{HDiffEq} into \eqref{LocErrLie} and exploit \eqref{Defphi2} to get
\[U(\tau)-U^\mathrm{Lie}_1 \approx e^{\tau A} \int_0^\tau (\tau-s) e^{-2s A} H(U_0) \dd s = \tau^2e^{\tau A}\varphi_2(-2\tau A)H(U_0).\] 
Adding this term on the Lie splitting \eqref{Lie} gives the formally second-order \emph{corrected Lie splitting}
\begin{equation}\label{CorrLieEasy}
U_{n+1} = \Psi_\tau(U_n) \coloneqq e^{\tau A}[U_n + \tau G(U_n) + \tau^2 \varphi_2(-2\tau A)H(U_n) ],
\end{equation}
which corresponds to \eqref{CorrLie}. From this derivation we immediately get the following representation of the local error. A related formula was derived in Lemma 6.2 of \cite{RuffSchnaubelt} in the 3D case.
\begin{Lemma}\label{LemmLokErrCorrLie}
	Let $U$ and $T$ be given by Assumption \eqref{AssWave}. Then the local error of the corrected Lie splitting \eqref{CorrLieEasy} is given by
	\[U(\tau)-U_1 
	= e^{\tau A} \int_0^\tau (\tau-s) e^{-2s A} \int_0^s e^{\sigma A}B(U(\sigma))\dd \sigma\dd s\]
	for all $\tau \in (0,T]$.
\end{Lemma}
\begin{proof}
	Follows directly from \eqref{LocErrLie} and the Duhamel formulation of \eqref{HDiffEq}.
\end{proof}

We can now bound the sum of local errors with the help of Corollary \ref{CorNullEstu}.

\begin{Lemma}\label{LemSumErrWave}
	Let $U = (u,\partial_tu)$, $T$, and $M$ be given by Assumption \ref{AssWave}.
	Then we can bound the sum of local errors of \eqref{CorrLieEasy} by
	\[ \Big\|\sum_{k=0}^{n-1} e^{(n-k-1)\tau A }\Big(U(t_{k+1})-\Psi_\tau(U(t_k))\Big) \Big\|_{H^1 \times L^2} \lesssim_{M,T} \tau^2,  \]
	for all $\tau \in (0,T]$ and $n \in \N_0$ with $n\tau \le T$.
\end{Lemma}
\begin{proof}
	By the triangle inequality and Lemma \ref{LemmLokErrCorrLie} with $U(t_k+\cdot)$ instead of $U$,
	\begin{align*}
	&\Big\|\sum_{k=0}^{n-1} e^{(n-k-1)\tau A }\Big(U(t_{k+1})-\Psi_\tau(U(t_k))\Big) \Big\|_{H^1 \times L^2} \\
	&\lesssim_T \tau^2\sum_{k=0}^{n-1} \int_0^\tau \|B(U(t_k+\sigma))\|_{H^1 \times L^2}\dd \sigma \le \tau^2 \|B(U)\|_{L^1([0,T],H^1 \times L^2)}.
	\end{align*}
	We next insert the definition \eqref{DefB} of $B$ and apply \eqref{eq:gbound} and finally Corollary \ref{CorNullEstu} to obtain
	\begin{align*}
	\|B(U)\|_{L^1([0,T],H^1 \times L^2)} &= \|g''(u)[(\partial_tu)^2-(\partial_xu)^2]+g'(u)g(u)\|_{L^1([0,T], L^2)} \\
	&\lesssim_{M,T} \|(\partial_tu)^2-(\partial_xu)^2\|_{L^2([0,T] \times \T)} + 1 \lesssim_{M,T} 1. \qedhere
	\end{align*}
\end{proof}

Similar as in the Schrödinger case, we conclude the proof of the global error bound using the discrete Gronwall lemma.

\begin{proof}[Proof of Theorem \ref{ThmCorrLie}]
	We proceed similar as in the proof of Theorem \ref{ThmLRI}.
	For the error
	\[ E_n \coloneqq U(t_n)-U_n\]
	we get the formula
	\begin{align*}
	E_n &= \sum_{k=0}^{n-1} e^{(n-k-1)\tau A}\Big( U(t_{k+1})-\Psi_\tau(U(t_k))\Big) \\ 
	&\quad+ \tau  \sum_{k=0}^{n-1} e^{(n-k)\tau A}\Big(G(U(t_k))-G(U_k)\Big) \\
	 &\quad+ \tau^2  \varphi_2(-2\tau A)\sum_{k=0}^{n-1} e^{(n-k)\tau A}\Big(H(U(t_k))-H(U_k)\Big).
	\end{align*}
	From Lemma \ref{LemSumErrWave}, \eqref{eq:gbound}, and standard estimates, we infer that
	\begin{equation*}\label{eq:waveproof}
	\|E_n\|_{H^1 \times L^2} \le c\tau^2 +  \tau\sum_{k=0}^{n-1}K(\|E_k\|_{H^1\times L^2})\|E_k\|_{H^1 \times L^2} 
	\end{equation*}
	with a constant $c>0$ and an increasing function $K$, both depending on $M$, $T$, and $L$. We define the maximum step size
	\[ \tau_0 \coloneqq (ce^{K(1)T})^{-\frac12}. \]
	Using the discrete Gronwall lemma, we obtain via induction on $n$ that
	\[ \|E_n\|_{H^1 \times L^2} \le c\tau^2 e^{K(1)T} \le 1 \]
	for all $ \tau \in (0,\tau_0]$ and $n \in \N_0$ with $n\tau \le T$.
\end{proof}

\section{Numerical Experiment}\label{SecNumExp}

{In this section we discuss the numerical behavior of the low-regularity integrators \eqref{LRI} and \eqref{CorrLie} applied to the nonlinear Schrödinger and wave equations \eqref{NLS} and \eqref{NLW2}, respectively. Our Python code to reproduce the results is available at \url{https://doi.org/10.35097/tqj4v6hysm9hvd9c}.}

\subsection{{Nonlinear Schrödinger equation}}
The numerical behavior of the scheme \eqref{LRI} applied to the nonlinear Schrödinger equation \eqref{NLS} has {already} been extensively studied in the literature. See, e.g., \cite{Alama,ORSBourgain,OS} for numerical experiments including comparisons with other schemes. However, they do {not} provide a clear picture concerning the convergence rate of the $L^2$ error in the situation of $H^1$ initial data. While the experiment in Figure 1 of \cite{Alama} shows first-order convergence as proven in our Theorem \ref{ThmLRI}, the experiment in Figure 1 of \cite{ORSBourgain} suggests an order reduction down to $3/4$. A possible explanation of this behavior is that an error bound of the form
\[ \|u(n\tau)-u_n\|_{L^2(\T)} \le c\tau^\beta \]
could hold for some $\beta \in (0,1)$, where one might have $c \ll C$, with $C$ from Theorem \ref{ThmLRI}, depending on the precise choice of initial data. Therefore, we provide a numerical example where a wider range of $\tau$ is considered than in \cite{Alama,ORSBourgain}.

We solve the nonlinear Schrödinger equation \eqref{NLS} with $\mu = 1$ and $T=1$ using the low-regularity integrator (LRI) \eqref{LRI}. As initial datum we use $u_0 = \phi/\|\phi\|_{L^2(\T)}$, where the function  $\phi \in H^1(\T)$ is defined by its Fourier coefficients
\begin{equation}\label{eq:ConstrInitialData}
	 \hat \phi_k = (1+|k|^2)^{-\frac12({\frac12+s}+\eps)}r_k
\end{equation}
for $k \in \{-K_0,\dots,K_0-1\}$, and $\hat\phi_k=0$ elsewhere. Here we use the {regularity index $s=1$}, the maximum frequency $K_0 = 2^{19}$, uniformly distributed numbers $r_k \in [-1,1]+\iu[-1,1]$, and a very small parameter $\eps>0$. The results are similar if we {set $r_k = 1$ (for all $k$)} instead. 
The space is discretized by the standard Fourier pseudo-spectral method, where we choose $K = 2^{11}$ grid points. The reference solution is computed using a Strang splitting with $\tau_{\mathrm{ref}} = 10^{-7}$. Higher values of $K$ and/or smaller values of $\tau_\mathrm{ref}$ did not change the outcome for the range of $\tau$ considered in the experiment below. 

In Figure \ref{fig}, we plot the maximal errors in the $L^2(\T)$ norm against the step sizes $\tau$.
We observe a convergence rate of approximately $3/4$ as in \cite{ORSBourgain} for the greater values of $\tau$ in the range $(10^{-3},10^{-2})$. For smaller values of $\tau$, the rate increases to approximately $9/10$.
In Table \ref{table}, we list the values of the step sizes $\tau_\ell$ and $L^2$ errors $e_\ell$, where the index $\ell \in \{1,\dots,10\}$ denotes the {corresponding} run of the experiment. Moreover, we compute the experimental order of convergence (EOC) by
\begin{equation}\label{eq:eoc}
\mathrm{EOC}_\ell = \frac{\log e_\ell-\log e_{\ell-1}}{\log \tau_\ell - \log \tau_{\ell-1}}. 
\end{equation}

\begin{figure}[h]
	\begin{tikzpicture}
		
		\definecolor{darkgray176}{RGB}{176,176,176}
		\definecolor{darkgreen}{RGB}{0,100,0}
		\definecolor{lightgray204}{RGB}{204,204,204}
		\definecolor{limegreen}{RGB}{50,205,50}
		
		\begin{axis}[
			legend cell align={left},
			legend style={
				fill opacity=0.8,
				draw opacity=1,
				text opacity=1,
				at={(0.03,0.97)},
				anchor=north west,
				draw=lightgray204
			},
			log basis x={10},
			log basis y={10},
			tick align=outside,
			tick pos=left,
			x grid style={darkgray176},
			xlabel={$\tau$},
			xmin=6.30957344480193e-07, xmax=0.0158489319246111,
			xmode=log,
			xtick style={color=black},
			y grid style={darkgray176},
			ylabel={$L^2$ error},
			ymin=2.305729521435e-05, ymax=0.217684787439361,
			ymode=log,
			ytick style={color=black}
			]
			\addplot [semithick, red, mark=x, mark size=3, mark options={solid}]
			table {%
				0.01 0.0884616179029136
				0.0035938 0.0413202866272042
				0.0012915 0.0191290805761551
				0.0004642 0.00839259839935793
				0.0001668 0.00346616832948106
				5.99e-05 0.00141577249001441
				2.15e-05 0.000570610789700885
				7.7e-06 0.0002259541893281
				2.8e-06 8.9753884019652e-05
				1e-06 3.49535326087069e-05
			};
			\addlegendentry{LRI \eqref{LRI}}
			\addplot [semithick, black, densely dotted]
			table {%
				0.01 0.0890221127607082
				0.0035938 0.0413202866272042
				0.0012915 0.0191786796524655
				0.0004642 0.00890280270871135
				0.0001668 0.00413185366393141
				5.99e-05 0.00191676121926721
				2.15e-05 0.000888846449796703
				7.7e-06 0.000411496306139365
				2.8e-06 0.000192693214058451
				1e-06 8.90221127607082e-05
			};
			\addlegendentry{$\mathcal O(\tau^{3/4})$}
			\addplot [semithick, black, dashed]
			table {%
				0.01 0.143596999589452
				0.0035938 0.0571667995906763
				0.0012915 0.0227578030699812
				0.0004642 0.00906108048568466
				0.0001668 0.00360679789248233
				5.99e-05 0.00143492600296312
				2.15e-05 0.000570610789700885
				7.7e-06 0.000226457661464577
				2.8e-06 9.11145204539565e-05
				1e-06 3.60699354874231e-05
			};
			\addlegendentry{$\mathcal O(\tau^{9/10})$}
		\end{axis}
	\end{tikzpicture}
	\caption{$L^2$ errors {of \eqref{LRI} applied to the nonlinear Schrödinger equation \eqref{NLS}.} The result of the experiment does not change if a more accurate spatial discretization and/or reference solution is used.}\label{fig}
\end{figure}

\begin{table}[h]
	\centering
	\caption{Experimental order of convergence in Figure \ref{fig} according to \eqref{eq:eoc}.}\label{table}
	\begin{tabular}{||c c c||} 
		\hline
		$\tau$ & $L^2$ error & EOC \\
		\hline\hline
		1e-02 & 8.84e-02 & --  \\ 
		3.59e-03 & 4.13e-02 & 0.74  \\
		1.29e-03 & 1.91e-02 & 0.75  \\
		4.64e-04 & 8.40e-03 & 0.81  \\
		1.67e-04 & 3.47e-03 & 0.86  \\  
		5.99e-05 & 1.42e-03 & 0.87  \\  
		2.15e-05 & 5.71e-04 & 0.89  \\  
		7.7e-06 & 2.26e-04 & 0.90  \\  
		2.8e-06 & 8.98e-05 & 0.91  \\  
		1e-06 & 3.50e-05 & 0.92  \\  
		\hline
	\end{tabular}
\end{table}

The results in Table \ref{table} indicate that the order of convergence is still increasing for very small values of $\tau$.
According to Theorem \ref{ThmLRI}, the convergence rate $1$ must show up for even smaller values of $\tau$. Unfortunately, we were unable to make this visible numerically since further experiments revealed that this would require a disproportionate computational effort.

\subsection{{Nonlinear wave equation}}
{In Figure 1 of \cite{KleinGordon}, the authors present a numerical example concerning the Sine--Gordon equation, i.e., the wave equation \eqref{NLW2} with nonlinearity $g(u)= -\sin(u)$.}
It shows second-order convergence in $H^1 \times L^2$ of the corrected Lie splitting \eqref{CorrLie} {for $H^1 \times L^2$ initial data} as predicted by our Theorem \ref{ThmCorrLie}. Comparisons with other schemes are also provided.

{Here, we present an analogous experiment for the defocusing cubic wave equation, which corresponds to \eqref{NLW2} with the different nonlinear function $g(u)=-u^3$. 
We solve this equation with the corrected Lie splitting (``Corr\-Lie'') \eqref{CorrLie} until time $T=1$. The initial data $(u_0,v_0) \in H^1 \times L^2$ are constructed using   \eqref{eq:ConstrInitialData} with $s=1$ and $s=0$, respectively, and scaled such that $\|u_0\|_{H^1} = \|v_0\|_{L^2} \approx 1$. 
The space is again discretized by the Fourier pseudo-spectral method, this time with $K = 2^{14}$ grid points. The reference solution is computed using \eqref{CorrLie} with $\tau_{\mathrm{ref}} = 10^{-5}$.} 

{In Figure \ref{fig2}, we plot the maximal errors in the $H^1(\T) \times L^2(\T)$ norm against various values of the step size $\tau$.
We observe that the order of convergence is precisely two, as predicted by our Theorem \eqref{ThmCorrLie} and as already observed in the case of the Sine--Gordon equation in \cite{KleinGordon}.}

\begin{figure}[h]
\begin{tikzpicture}
	
	\definecolor{darkgray176}{RGB}{176,176,176}
	\definecolor{darkgreen}{RGB}{0,100,0}
	\definecolor{lightgray204}{RGB}{204,204,204}
	\definecolor{limegreen}{RGB}{50,205,50}
	
	\begin{axis}[
		legend cell align={left},
		legend style={
			fill opacity=0.8,
			draw opacity=1,
			text opacity=1,
			at={(0.03,0.97)},
			anchor=north west,
			draw=lightgray204
		},
		log basis x={10},
		log basis y={10},
		tick align=outside,
		tick pos=left,
		x grid style={darkgray176},
		xlabel={$\tau$},
		xmin=7.94328234724282e-05, xmax=0.0125892541179417,
		xmode=log,
		xtick style={color=black},
		y grid style={darkgray176},
		ylabel={$H^1 \times L^2$ error},
		ymin=4.89943103706789e-09, ymax=0.000264249162289053,
		ymode=log,
		ytick style={color=black}
		]
		\addplot [semithick, red, mark=x, mark size=3, mark options={solid}]
		table {%
			0.01 7.82353826233572e-05
			0.00785 4.8270889734632e-05
			0.00616 2.97281108876031e-05
			0.00483 1.83114013719391e-05
			0.00379 1.12944224762412e-05
			0.00298 6.99475615966326e-06
			0.00234 4.31865208717661e-06
			0.00183 2.64661126273916e-06
			0.00144 1.64166113561958e-06
			0.00113 1.01277975521301e-06
			0.00089 6.29432794028145e-07
			0.0007 3.90209663991512e-07
			0.00055 2.41388289252324e-07
			0.00043 1.4785653015372e-07
			0.00034 9.26343382245344e-08
			0.00026 5.42850627720861e-08
			0.00021 3.54639074151668e-08
			0.00016 2.06107269874879e-08
			0.00013 1.36077217547667e-08
			0.0001 8.03952808656627e-09
		};
		\addlegendentry{CorrLie \eqref{CorrLie}}
		\addplot [semithick, black, dashed]
		table {%
			0.01 0.000161038127275346
			0.00785 9.92357199802501e-05
			0.00616 6.11068836233937e-05
			0.00483 3.75684236739382e-05
			0.00379 2.3131677639958e-05
			0.00298 1.43008298545598e-05
			0.00234 8.81780369708885e-06
			0.00183 5.39300584432406e-06
			0.00144 3.33928660718158e-06
			0.00113 2.05629584717889e-06
			0.00089 1.27558300614802e-06
			0.0007 7.89086823649196e-07
			0.00055 4.87140335007922e-07
			0.00043 2.97759497332115e-07
			0.00034 1.861600751303e-07
			0.00026 1.08861774038134e-07
			0.00021 7.10178141284276e-08
			0.00016 4.12257605824886e-08
			0.00013 2.72154435095335e-08
			0.0001 1.61038127275346e-08
		};
		\addlegendentry{$\mathcal O(\tau^2)$}
	\end{axis}
	
\end{tikzpicture}
\caption{{$H^1 \times L^2$ errors of \eqref{CorrLie} applied to the nonlinear wave equation \eqref{NLW2} with $g(u)=-u^3$. The result of the experiment does not change if a more accurate spatial discretization} {and/or reference solution is used.}}\label{fig2}
\end{figure}

\printbibliography

@ARTICLE{ORS,
	AUTHOR = {Ostermann, A. AND Rousset, F. AND Schratz, K.},
	TITLE = {Error estimates of a Fourier integrator for the cubic Schrödinger equation at low regularity},
	JOURNAL = {Found. Comput. Math.},
	VOLUME  = {21},
	NUMBER  = {},
	YEAR = {2021},
	PAGES = {725--765}
}

@ARTICLE{KdVLowReg,
	AUTHOR = {Hofmanová, M. AND Schratz, K.},
	TITLE = {An exponential-type integrator for the KdV equation},
	JOURNAL = {Numer. Math.},
	VOLUME  = {136},
	NUMBER  = {},
	YEAR = {2017},
	PAGES = {1117--1137}
}

@ARTICLE{KdVUnfiltered,
	AUTHOR = {Li, B. AND Wu, Y.},
	TITLE = {An unfiltered low-regularity integrator for the KdV equation with solutions below $H^1$},
	JOURNAL = {Found. Comput. Math.},
	%VOLUME  = {136},
	%NUMBER  = {},
	YEAR = {2025},
	DOI = {https://doi.org/10.1007/s10208-025-09702-0},
	%PAGES = {1117--1137}
}

@ARTICLE{OS,
	AUTHOR = {Ostermann, A. AND Schratz, K.},
	TITLE = {Low regularity exponential-type integrators for semilinear Schrödinger equations},
	JOURNAL = {Found. Comput. Math.},
	VOLUME  = {18},
	NUMBER  = {},
	YEAR = {2018},
	PAGES = {731--755}
}

@ARTICLE{RS,
	AUTHOR = {Rousset, F. AND Schratz, K.},
	TITLE = {A general framework of low regularity integrators},
	JOURNAL = {SIAM J. Numer. Anal.},
	VOLUME  = {59},
	NUMBER  = {3},
	YEAR = {2021},
	PAGES = {1735--1768}
}

@ARTICLE{IgnatSplitting,
	AUTHOR = {Ignat, L.},
	TITLE = {A splitting method for the nonlinear Schrödinger equation},
	JOURNAL = {J. Differential Equations},
	VOLUME  = {250},
	NUMBER  = {},
	YEAR = {2011},
	PAGES = {3022--3046}
}

@ARTICLE{ExpInt,
	AUTHOR = {Hochbruck, M. AND Ostermann, A.},
	TITLE = {Exponential integrators},
	JOURNAL = {Acta Numerica},
	VOLUME  = {19},
	NUMBER  = {},
	YEAR = {2010},
	PAGES = {209--286}
}

@ARTICLE{Gauckler,
	AUTHOR = {Gauckler, L.},
	TITLE = {Error analysis of trigonometric integrators for semilinear Klein-Gordon equations},
	JOURNAL = {SIAM J. Numer. Anal.},
	VOLUME  = {53},
	NUMBER  = {},
	YEAR = {2015},
	PAGES = {1082--1106}
}

@ARTICLE{Lubich,
	AUTHOR = {Lubich, C.},
	TITLE = {On splitting methods for Schrödinger-Poisson and cubic nonlinear Schrödinger equations},
	JOURNAL = {Math. Comp.},
	VOLUME  = {77},
	NUMBER  = {},
	YEAR = {2008},
	PAGES = {2141--2153}
}

@BOOK{Tao2006,
	AUTHOR = {Tao, T.},
	YEAR = {2006},
	TITLE = {Nonlinear Dispersive Equations: Local and Global Analysis},
	EDITION = {},
%	ISBN = {978-0-821-84143-3},
	PUBLISHER = {American Mathematical Society},
	ADDRESS = {Providence, RI},
}

@ARTICLE{ORSBourgain,
	AUTHOR = {Ostermann, A. AND Rousset, F. AND Schratz, K.},
	TITLE = {Fourier integrator for periodic NLS: low regularity estimates via discrete Bourgain spaces},	
	JOURNAL = {J. Eur. Math. Soc.},
	VOLUME  = {25},
	NUMBER  = {10},
	YEAR = {2023},
	PAGES = {3913--3952}
}

@ARTICLE{KleinGordon,
	AUTHOR = {Li, B. AND Schratz, K. AND Zivcovich, F.},
	TITLE = {A second-order low-regularity correction of Lie splitting for the semilinear Klein--Gordon equation},	
	JOURNAL = {ESAIM Math. Model. Numer. Anal.},
	VOLUME  = {57},
	NUMBER  = {2},
	YEAR = {2023},
	PAGES = {899--919}
}

@ARTICLE{RuffSchnaubelt,
	AUTHOR = {Ruff, M. AND Schnaubelt, R.},
	TITLE = {Error analysis of the Lie splitting for semilinear 
	wave equations with finite-energy solutions},
	JOURNAL = {Discrete Contin. Dyn. Syst.},
	VOLUME  = {45},
	NUMBER  = {9},
	YEAR = {2025},
	PAGES = {2969--3008}	
}

@ARTICLE{DiscSol,
	AUTHOR = {Cao, J. AND Li, B. AND Lin, Y. AND Yao, F.},
	TITLE = {Numerical approximation of discontinuous solutions of the semilinear wave equation},	
	JOURNAL = {SIAM J. Numer. Anal.},
	VOLUME  = {63},
	NUMBER  = {1},
	YEAR = {2025},
	PAGES = {214--238}
}

@ARTICLE{Zygmund,
	AUTHOR = {Zygmund, A.},
	TITLE = {On Fourier coefficients and transforms of functions of two variables},
	JOURNAL = {Studia Math.},
	VOLUME  = {50},
	NUMBER  = {},
	YEAR = {1974},
	PAGES = {189–-201}
}

@ARTICLE{Bourgain93,
	AUTHOR = {Bourgain, J.},
	TITLE = {Fourier transform restriction phenomena for certain lattice subsets and applications to nonlinear evolution equations I: Schrödinger equations},
	JOURNAL = {Geom. Funct. Anal.},
	VOLUME  = {3},
	NUMBER  = {},
	YEAR = {1993},
	PAGES = {107–-156}
}

@ARTICLE{KlainermanMachedon,
	AUTHOR = {Klainerman, S. and Machedon, M.},
	TITLE = {Space-time estimates for null forms and the local existence theorem},
	JOURNAL = {Comm. Pure Appl. Math.},
	VOLUME  = {46},
	NUMBER  = {},
	YEAR = {1993},
	PAGES = {1221–-1268}
}

@ARTICLE{CaoSchr,
	AUTHOR = {Cao, J. AND Li, B. AND Lin, Y.},
	TITLE = {A new second-order low-regularity integrator for the cubic nonlinear Schrödinger equation},
	JOURNAL = {IMA J. Numer. Anal.},
	VOLUME  = {44},
	NUMBER  = {},
	YEAR = {2024},
	PAGES = {1313–-1345}
}

@ARTICLE{OWY,
	AUTHOR = {Ostermann, A. AND Wu, Y. AND Yao, F.},
	TITLE = {A second-order low-regularity integrator for
	the nonlinear Schrödinger equation},
	JOURNAL = {Adv. Contin. Discrete Models},
	VOLUME  = {},
	NUMBER  = {23},
	YEAR = {2022},
	PAGES = {}
}

@ARTICLE{Alama,
	AUTHOR = {Alama Bronsard, Y.},
	TITLE = {A symmetric low-regularity integrator for the nonlinear Schrödinger equation},
	JOURNAL = {IMA J. Numer. Anal.},
	VOLUME  = {44},
	NUMBER  = {6},
	YEAR = {2024},
	PAGES = {3648–-3682}
}

@ARTICLE{BrunedSchratz,
	AUTHOR = {Bruned, Y. AND Schratz, K.},
	TITLE = {Resonance-based schemes for dispersive equations via decorated trees},
	JOURNAL = {Forum Math. Pi},
	VOLUME  = {10},
	NUMBER  = {},
	YEAR = {2022},
	PAGES = {1–-76}
}

@ARTICLE{MaierhoferExplSym,
	AUTHOR = {Feng, Y. AND Maierhofer, G. AND Wang, C.},
	TITLE = {Explicit symmetric low-regularity integrators for the nonlinear Schrödinger equation},
	JOURNAL = {SIAM J. Sci. Comput.},
	VOLUME  = {47},
	NUMBER  = {4},
	YEAR = {2025},
	PAGES = {A2154–-A2179}
}

@MISC{R,
	AUTHOR = {Ruff, M.},
	TITLE = {Error analysis of the Strang splitting for the 3D semilinear wave equation with finite-energy data},
	NOTE = {arXiv preprint},
	DOI = {https://doi.org/10.48550/arXiv.2503.13126},
	YEAR = {2025}
}

\end{document}